\documentclass[12pt]{amsart}
\usepackage{amssymb}
\usepackage{amsmath}
\usepackage{amsthm} 
\usepackage{graphicx,subcaption}
\usepackage{stmaryrd}
\usepackage{physics}
\usepackage[utf8]{inputenc}
\usepackage{amsfonts}
\usepackage{graphicx}
\usepackage{hyperref}
\usepackage{xcolor}

\hypersetup{
	colorlinks   = true, 
	urlcolor     = blue, 
	linkcolor    = blue, 
	citecolor   = blue 
}

\numberwithin{equation}{section}
\newtheorem{theorem}{Theorem}[section]

\newtheorem{lemma}[theorem]{Lemma}
\newtheorem*{theorem*}{Theorem}
\newtheorem*{lemma*}{Lemma}
\newtheorem{claim}[theorem]{Claim}

\newtheorem{problem}[theorem]{Problem}\newtheorem{example}[theorem]{Example}
 \newtheorem{remark}[theorem]{Remark}\newtheorem{proposition}[theorem]{Proposition}\newtheorem{definition}[theorem]{Definition}
\newtheorem{corollary}[theorem]{Corollary} 
\newcommand{\R}{{\mathbb R}}  \newcommand{\Z}{{\mathbb Z}} \newcommand{\N}{{\mathbb N}}

\newcommand{\Cc}{{\mathbb C}}\newcommand{\cc}{{\bf c}}   
\newcommand{\sinc}{{\rm sinc}} 
\newcommand{\Zer}{{\rm Zer}} 

\DeclareMathOperator*{\essi}{ess\,inf}
\DeclareMathOperator*{\ess}{ess\,sup}

\newcommand{\Ff}{\mathcal{G}}
\newcommand{\supp}{{\rm supp\,}}

\title[Stability of shifts ]{Stability of shifts,  interpolation, and crystalline measures}
\thanks{I. Z. gratefully acknowledges support from the Austrian Science Fund (FWF) [\href{https://doi.org/10.55776/P33217}{10.55776/P33217}] and from the Research Council of Norway by Grant 334466, “Fourier Methods
and Multiplicative Analysis”}

 \author{Alexander Ulanovskii}
 \address{Department of Mathematics and Physics, University of Stavanger, 4036 Stavanger, Norway}
 \email{alexander.ulanovskii@uis.no} 
 \author{Ilya Zlotnikov}
 \address{ Faculty of Mathematics, University of Vienna, Oskar-Morgenstern-Platz 1,	A-1090 Vienna, Austria}
  \address{
  Department of Mathematical Sciences, Norwegian University of Science and Technology (NTNU), 7491 Trondheim, Norway
  }
\email{ilia.k.zlotnikov@ntnu.no}



\begin{document}

\begin{abstract}

Let $V^p_\Gamma(\Ff),1\leq p\leq\infty,$ be the quasi shift-invariant space generated by $\Gamma$-shifts of a  function $\Ff$, where $\Gamma\subset\R$ is a separated set. For several large families of generators $\Ff$, we present necessary and sufficient conditions on $\Gamma$ that imply that the $\Gamma$-shifts of $\Ff$ form an unconditional basis for $V^p_\Gamma(\Ff)$. The connection between this property, interpolation, universal interpolation, and crystalline measures is discussed.
\end{abstract}
 \date{}\maketitle

\section{Introduction}

\subsection{Paley--Wiener and Wiener amalgam spaces}Let $\hat f$ denote  the  Fourier transform of a function $f$,
$$\hat f(x)=\int_\R e^{-2\pi i xt}f(t)\,dt.$$

 Given a set $S\subset\R$ of finite positive measure, the classical Paley--Wiener space $PW_S$ consists of all $f\in L^2(\R)$ whose inverse Fourier transform vanishes a.e. outside $S$.

The Wiener amalgam space $W$ consists of measurable functions $f: \R \to \Cc:$ satisfying
\begin{equation}\label{wiener}\|f\|_W := \sum\limits_{k \in \Z} \|f\|_{L^\infty (k,k+1)} < \infty.\end{equation}

The space $W$ introduced by N.~Wiener (see \cite{MR1503035}) is proven to be important in harmonic analysis and time-frequency analysis, particularly in Gabor frame theory and theory of shift and quasi-shift invariant spaces, see e.g., \cite{MR1756138,MR2264211,MR1843717}, and references therein.

A trivial observation frequently used below is that condition $\Ff\in W$ implies $\Ff\in L^p(\R), 1\leq p\leq\infty,$ and that  $\hat{\Ff}$ is a continuous function.

\subsection{Quasi shift-invariant spaces}
A  set $\Gamma\subset\R$ is called separated  if 
\begin{equation}\label{sepa}
\Delta(\Gamma) := \inf_{\gamma,\gamma'\in\Gamma, \gamma \neq \gamma'}|\gamma-\gamma'|>0.\end{equation}
Every separated set is countable.

Given  a  function $\Ff:\R\to \Cc$ with a "reasonably fast" decay at $\pm\infty$ and a number  $p,1\leq p\leq\infty$, the shift-invariant space $V_\Z^p(\Ff)$ consists of all functions $f$ of the form 
$$f(x)=\sum_{n\in\Z}c_n \Ff(x-n),\quad \{c_n\}\in l^p(\Z).$$ 

More generally, given a separated set $\Gamma\subset\R$, the quasi shift-invariant space $V_\Gamma^p(\Ff)$ consists of all functions of the form 
\begin{equation}\label{f0}f(x)=\sum_{\gamma\in\Gamma}c_\gamma \Ff(x-\gamma),\quad \{c_\gamma\}\in l^p(\Gamma).\end{equation}
The function $\Ff$ is called the generator for the space $V_\Gamma^p(\Ff)$.

A classical example is the Paley--Wiener space $PW^2_{[-1/2,1/2]}$ which is exactly the shift-invariant space $V^2_\Z(\Ff)$ generated by the sinc-function\begin{equation}\label{sinc}
  \Ff(x)=\sinc(x):=\frac{\sin(\pi x)}{\pi x}.  
\end{equation} 

The (quasi) shift-invariant spaces have important applications in mathematics and engineering, particularly because they are often used as models for spaces of signals and images. For example, in time-frequency analysis, the well-known close connection between the Gabor frames and sampling sets for shift-invariant spaces allows one to obtain new information on the frame sets for different generators, see \cite{grs, MR4047939, MR3053565,  MR4782146}.

\subsection{Main problem}
Let $1 \le p \le\infty.$ We study a basic property of the quasi shift-invariant space $V^p_\Gamma(\Ff)$, specifically  that the $\Gamma$-shifts of $\Ff$ are $l^p$-stable, i.e. there exist positive constants $C_1$ and $C_2$ such that
\begin{equation}\label{stab}C_1\|\cc\|_p \le \left\|\sum_{\gamma \in \Gamma} c_{\gamma} \Ff(\,\cdot - \gamma) \right\|_p \le C_2 \|\cc\|_p,\quad \mbox{for every $\cc=\{c_\gamma\} \in l^p(\Gamma).$}\end{equation}

 The $l^p$-stability means that $V^p_\Gamma(\Ff)$ is a closed subspace of $L^p(\R)$ and that the system $\{\Ff(\,\cdot - \gamma)\}_{\gamma \in \Gamma}$ forms an unconditional basis in this space. For $p=2$,  condition \eqref{stab} means that the system $\{\Ff(\,\cdot - \gamma)\}_{\gamma \in \Gamma}$ constitutes a Riesz basis for  $V^2_{\Gamma}(\Ff).$

 In this paper, we consider the following

\begin{problem}  What assumptions should be imposed on  $\Gamma$ and  $\Ff$ to ensure that the set of $\Gamma$-shifts of $\Ff$ is {\it $l^p$-stable}?
\end{problem}

The stability property of $\Z$-shifts is well-studied. 
The following result is an immediate corollary of Theorem~3.5  in  \cite{Jia1991}:

\begin{theorem}\label{integershifs}
    Assume $\Ff\in W$ and $p\in[1,\infty]$. Then the integer-shifts of $\Ff$ are $l^p$-stable if and only if the Fourier transform $\hat\Ff$ of $\Ff$ does not vanish on any set $\Z+b, 0\leq b<1. $
   \end{theorem}

However,   results on the stability of $\Gamma$-shifts are scarce.   We mention the paper \cite{MR3762092}, which proves the $l^2$-stability of   $\Gamma$-shifts for certain generators $\Ff$ under the condition that $\Gamma$ is a complete interpolating sequence for the Paley--Wiener space  $PW_{[-1/2,1/2]}$.
We also mention the papers \cite{MR3053565,uz}  which, in particular, study the stability of $\Gamma$-shifts for generators from certain classes. In particular, in \cite{MR3053565}, it was shown that $\Gamma$-shifts of totally positive generators are $l^2$-stable for every separated set $\Gamma.$

  We will say that the $\Gamma$-shifts of $\Ff$ are stable if they are $l^p$-stable for every $1\le p\le\infty.$ We note that for a wide class of generators, the stability property of shifts does not depend on $p$.
Observe that the right-hand-side inequality in \eqref{stab} is true for every separated set $\Gamma$, every generator $\Ff\in W$ and every $p\in[1,\infty]$, see Lemma 8.1 in \cite{uz}. On the other hand, under a mild additional condition on the generator, it suffices to check that the left-hand-side inequality is true for   $p=\infty$. 
The following result follows from Lemma 8.1 and Theorem 1.6 in \cite{uz}:

\begin{theorem}\label{t0}
    Let $\Gamma$ be a separated set and $\Ff\in W\cap C(\R)$. If  $\Gamma$-shifts of $\Ff$ are $l^\infty$-stable, then they are stable. 
\end{theorem}

In a more general situation, namely, $\Ff \in W,$ the conclusion of this theorem  remains true if one replaces $p=\infty$ with $p=2$, see Theorem~\ref{t1} below. Moreover, under additional assumptions on $\Ff$, the $l^{\infty}$-space can be replaced by $c_0$-space or even $c_{00}$-space, see Theorem~\ref{schwartz_approx_theorem}.

We conclude this section with a brief presentation of our two main results.
They show an intimate connection between the stability of general shifts, a certain interpolation property, and the existence of certain crystalline measures. 

\begin{itemize}
    \item Theorem \ref{t3} (see  sec. 2) states that, if $|\hat\Ff|^2\in W$, then $\Gamma$-shifts of $\Ff$ are $l^2$-stable if and only if there is a number $r>0$ such that $\Gamma$ is an interpolation set for $PW_{\{t\in\R: |\hat\Ff(t)|>r\}}$.

    \item Theorem \ref{t5}  (see sec. 3) states that, if $\Ff\in W\cap C(\R)$ and the zero set ${\rm Zer}(\hat\Ff)$ of $\hat{\Ff}$ is locally finite,  then $\Gamma$-shifts of $\Ff$ are $l^{\infty}$-stable if and only no sequence of translates $\Gamma-x_k,x_k\in\R,$ may weakly converge to a set that supports a sparse crystalline measure with bounded coefficients whose spectrum lies in ${\rm Zer}(\hat\Ff).$
\end{itemize}

See Sections~2 and 3 for necessary definitions and some results on interpolation in Paley--Wiener spaces and crystalline measures. 
Using these results and the above theorems, we obtain new results on the stability of shifts: For the generators $\Ff$ satisfying $|\Ff|^2\in W$ and $\Zer{\,(\hat{\Ff})} = \R \setminus (a,b), \, a,b \in \R$, we completely describe the family of separated sets $\Gamma$ such that $\Gamma$-shifts of $\Ff$ are stable.  A similar result is obtained for the generators $\Ff\in W\cap C(\R)$ such that $\Zer{\,(\hat{\Ff})}$ is a finite union of separated sets, see  Sections~2 and 3. 

\section{\texorpdfstring{$l^2$}{l2}-stability of \texorpdfstring{$\Gamma$}{Gamma}-shifts and interpolation}\label{Sec_stab_interpol}

\subsection{}
 When $\Gamma=\Z,$ one may describe the set of all generators for which 
\eqref{stab}  with $p=2$ is true. The following result is known (see e.g., \cite[Theorem~10.19]{MR2744776}) and not hard to prove: 

\begin{lemma}\label{Integer_rhs_cond} Given a non-trivial generator $\Ff\in L^2(\R)$. For every non-trivial sequence ${\bf c}=\{c_n\}\in l^2(\Z)$, we have 
$$  \essi_{t\in(0,1)}\sum\limits_{n \in \Z} |\hat{\Ff} (t -n)|^2\leq \frac{1}{\|{\bf c}\|_2^2 }\left\|\sum_{n\in\Z}c_n \Ff(t-n)\right\|_2^2\leq \ess_{t\in(0,1)} \sum\limits_{n \in \Z} |\hat{\Ff} (t -n)|^2.$$Both inequalities in \eqref{eqstab} are sharp, with the right-hand side being the smallest possible value and the left-hand side the largest possible value that satisfies the inequalities.
\end{lemma}

Therefore, the conditions
\begin{equation}\label{eqstab}\essi_{t\in(0,1)} \sum\limits_{n \in \Z} |\hat{\Ff} (t -n)|^2>0,\quad  \ess_{t\in(0,1)} \sum\limits_{n \in \Z} |\hat{\Ff} (t -n)|^2<\infty\end{equation}are necessary and sufficient for the $l^2$-stability of $\Z$-shifts.
See Section~\ref{sect_lemma_above_proof} for an extension of the right-hand-side inequality in Lemma \ref{Integer_rhs_cond}  to a certain class of separated sets $\Gamma.$

In this section, we will consider generators $\Ff$ satisfying the condition 
\begin{equation}\label{e1}
    |\hat\Ff |^2\in W.
\end{equation}
 Condition~\eqref{e1} has already appeared in studying the stability properties of integer translates, see \cite[Lemma~10.24]{MR2744776}.
In particular, it implies the second  inequality in \eqref{eqstab}. It also implies $\hat\Ff\in L^2(\R), $ and so $\Ff\in L^2(\R).$ This condition is sufficient for the validity of the right-hand-side inequality in \eqref{stab}  with $p=2$:

\begin{lemma}\label{lemma_above_est}
       Assume $\Gamma$ is a separated set and $\Ff$ satisfies \eqref{e1}. There exists $C>0$ such that
        \begin{equation}\label{l2stab_above}
            \left\|\sum_{\gamma \in \Gamma} c_{\gamma} \Ff(\,\cdot - \gamma) \right\|^2_2 \le C \|\cc\|_2^2,\quad \mbox{for every $\cc=\{c_\gamma\} \in l^2(\Gamma).$}
        \end{equation}
\end{lemma}

\subsection{Stability of shifts and interpolation sets for Paley-Wiener spaces}  
\begin{definition}  Let  $S\subset\R$ be a set of positive finite measures. A separated set $\Gamma$ is called an interpolation set for $PW_S$ if for every $\{c_\gamma\}\in l^2(\Gamma)$ there exists $f\in PW_S$ satisfying $f(\gamma)=c_\gamma$ for all $\gamma\in\Gamma.$
\end{definition}

It is well-known that every interpolation set $\Gamma$ for a space  $PW_S$ is separated.

When $S=[a,b]$ is a single interval, the interpolation property can be essentially described in terms of the upper uniform density of $\Gamma$, defined by $$D^+(\Gamma):=\lim_{r\to\infty}\sup_{x\in\R}\frac{\#(\Gamma\cap(x,x+r))}{r}.$$

\begin{theorem}[\cite{MR102702}]\label{inter}
 If $D^+(\Gamma)<b-a,$ then $\Gamma$ is an interpolation set for $PW_{[a,b]}.$ If $D^+(\Gamma)>
b-a,$ then it is not. 
\end{theorem}

A complete description of interpolation sets for the case of a single interval is known, but it is not expressed in terms of some density (see \cite{MR2040080}). For a discussion on the interpolation problem for $PW_S$ in the setting of disconnected spectra $S$, we refer the reader to \cite[Lecture~5]{ou}.

We now formulate the main result of this section.
\begin{theorem}\label{t3}
    Assume $\Ff$ is a non-trivial generator satisfying \eqref{e1}.
   Then $\Gamma$-shifts of $\Ff$ are $l^2$-stable if and only if there is a number $r>0$ such that $\Gamma$ is an interpolation set for $PW_{\{|\hat\Ff|>r\}}$.   
    \end{theorem}
Above we set $\{|\hat\Ff|>r\}:=\{t\in\R:|\hat\Ff(t)|>r\}$. This set is defined a.e.

Using Theorem \ref{t3}, we give two examples of generators $\Ff$ for which it is easy to describe all separated sets $\Gamma$ such that  $\Gamma$-shifts of $\Ff$ are  $l^2$-stable.

\begin{example} Let $\Ff(x)=\sinc(x)$. Then $\hat\Ff={\bf 1}_{[-1/2,1/2]}$. By  Theorem~\ref{t3},   $\Gamma$-shifts of $\Ff$ are $l^2$-stable if and only if $\Gamma$ is an interpolation set for $PW_{[-1/2,1/2]}$. Therefore, the interpolation sets $\Gamma$ for $PW_{[-1/2,1/2]}$ described in \cite{MR2040080} are exactly the sets $\Gamma$ for which $l^2$-shifts of $\Ff(x)=\sinc(x)$ are $l^2$-stable.\end{example}

A similar result holds for every generator $\Ff$  such that 
 $|\hat\Ff(t)|>\delta> 0$ a.e. on an interval $(a,b)$ and $\hat \Ff(t)=0$ a.e. for $t\not\in (a,b)$.

\begin{example} Let $\Ff(x)={\rm sinc}^n(x), n\geq2$. Then $\hat\Ff$ is continuous and vanishes outside $(-n/2,n/2)$. By Theorems \ref{inter} and \ref{t3},  $\Gamma$-shifts of $\Ff$ are $l^2$-stable if and only if $D^+(\Gamma)<n$.  
\end{example}

More generally, if   $\hat\Ff(t)$ is a continuous function whose  zero set is $\Zer(\hat{\Ff})= \R\setminus(a,b)$, then  $\Gamma$-shifts of $\Ff$ are $l^2$-stable if and only if $D^+(\Gamma)<b-a$.

Finally, the following  result is a simple consequence of  Theorems \ref{inter} and \ref{t3}:

\begin{corollary}\label{t2} The following statements are equivalent:

{\rm(i)} $D^+(\Gamma)=0${\rm;} 

{\rm(ii)}  $\Gamma$-shifts of $\Ff$ are  $l^2$-stable, for every non-trivial $\Ff$ satisfying  \eqref{e1} and $\hat\Ff\in C(\R)$.
\end{corollary}

\subsection{ \texorpdfstring{$l^2$-stability}{l2} and universal interpolation}

\begin{definition} A separated set
$\Gamma$ is called a universal  interpolation set  if it is an interpolation set for $PW_S$, for every open set $S$  satisfying $|S|< D^+(\Gamma),$ where $|S|$ is the measure of $S.$
\end{definition}

The first construction of universal interpolation (and universal sampling) sets was presented in \cite{MR2439002}.
For another construction, see  \cite{MR2674875}.

If $\Gamma$ is a universal interpolation set, the stability of $\Gamma$-shifts  depends only on the size of spectrum of $\Ff.$ 
More precisely, from Theorem \ref{t3} we deduce
\begin{corollary}
    The following statements are equivalent:

    {\rm(i)} $\Gamma$ is a universal interpolation set{\rm;}

    {\rm(ii)} If  $\Ff$  satisfies \eqref{e1} and $\hat\Ff\in C(\R)$, then $\Gamma$-shifts of $\Ff$ are $l^2$-stable if and only if $|\{|\hat\Ff|>0\}|>D^+(\Gamma)$.
\end{corollary}

\subsection{From  \texorpdfstring{$l^2$}{l2}- to  \texorpdfstring{$l^p$}{lp}-stability} Somewhat similar to  Theorem \ref{t0}, we prove
\begin{theorem}\label{t1}
     Assume $\Ff\in W$. If $\Gamma$-shifts of $\Ff$ are $l^2$-stable, then they are stable.
\end{theorem}

\section{Stability of shifts and crystalline measures}\label{Sec_stab_quasicrystal}
Below we consider the family of continuous generators $\Ff\in W$ such that the
zero set  Zer$(\hat\Ff):=\{t:\hat\Ff(t)=0\}$ of  $\hat\Ff$ is {\it locally finite}, i.e.
\begin{equation}\label{locfin}
     \#\left(\mbox{Zer}(\hat\Ff) \cap (x,x+1)\right) < \infty,\quad  \mbox{ for every } x\in\R.
\end{equation}
This condition is satisfied whenever $\hat\Ff$ is a real analytic (or a quasi-analytic) function.
In particular, every  $\Ff\in L^1(\R)$ with bounded support satisfies \eqref{locfin}.

It turns out that the stability of $\Gamma$-shifts property for the generators satisfying \eqref{locfin} is equivalent to the existence of certain crystalline measures.

\subsection{Sparse crystalline measures} A crystalline measure is an atomic measure supported by a locally ﬁnite set whose distributional Fourier transform is also supported by a locally ﬁnite set. More precisely, following Yves Meyer {\rm(}see e.g. \cite{MR3482845} or \cite{Meyer_TN}{\rm)}, we introduce

\begin{definition}
An  atomic measure $$\mu=\sum_{\gamma\in\Gamma}c_\gamma\delta_\gamma$$is a {\it crystalline measure} if 
\begin{enumerate}
    \item[\bf (i)] $\mu$ is a tempered distribution{\rm;}
    \item[\bf (ii)] The distributional Fourier transform $\hat\mu$ is again an atomic measure $$\hat\mu=\sum_{s\in S}a_s\delta_s;$$
    \item[\bf (iii)] The support $\Gamma$ of $\mu$ and its spectrum  $S$  are locally finite sets.
\end{enumerate}

Measure $\mu$ is called a {\it sparse crystalline measure} if the set $\Gamma$ is separated.
\end{definition}

Denote by $S(\R)$ the set of all Schwartz functions. Given a crystalline measure $\mu$, the following generalized Poisson formula holds for every  $F\in S(\R)$ (see \cite{Meyer_TN}, Section 3): 
\begin{equation}\label{crys}
    \sum_{\gamma\in\Gamma}c_\gamma \hat F(\gamma)=\sum_{s\in S}a_s F(s).
\end{equation}

A classical example of a sparse crystalline measure is the Poisson comb $$\mu=\sum_{n\in\Z}\delta_n.$$The Poisson summation formula implies $\hat\mu=\mu.$
More generally, every generalized Poisson comb \begin{equation}\label{poi}
    \mu=\sum_{k=1}^m\sum_{n\in\Z}P_k(n)\delta_{an+x_k},\quad m\in\N,\end{equation}is a sparse crystalline measure, where $a>0$, $x_k\in\R$ and $P_k, k=1,...,m,$ are arbitrary exponential polynomials with purely imaginary frequencies. One can check that the Fourier transform of a generalized Poisson comb is also a generalized Poisson comb.

Crystalline measures are useful for studying mathematical structures and their applications in signal processing and crystallography.
They are essential for finding an adequate mathematical model for the physical quasicrystals discovered in the 1980s by  Dan Shechtman.

 We will say that a separated set $\Gamma$ supports a sparse crystalline measure, if there exists a sparse crystalline measure $\mu$ whose support lies in $\Gamma$.

\subsection{Weak limits  of translates of  separated sets}
Given a sequence of separated sets $\Gamma_k\subset\R$ such that $\inf_k \Delta(\Gamma_k) > 0,$ where $\Delta(\Gamma)$ is defined in \eqref{sepa}, we say that the sequence converges weakly to some separated set  $\Gamma' \subset \R$ if for every $\epsilon>0$ and $R>0$ there is an integer $l=l(\epsilon, R)$ such that
$$\Gamma_k\cap(-R,R) \subset \Gamma'+(-\epsilon, \epsilon),\quad \Gamma'\cap(-R,R)\subset \Gamma_k+(-\epsilon, \epsilon), \quad k\geq l. $$

Let $\Gamma\subset\R$ be a separated set. We denote by $W(\Gamma)$ the set of weak limits of all possible sequences $\Gamma_k:=\Gamma-x_k, x_k\in\R.$ 

We need the following lemma, see e.g., \cite[Lecture~3]{ou}.
\begin{lemma}\label{lsep} Every sequence of separated sets 
$\{\Gamma_k\}_{k\in\N}$ satisfying $\inf_k \Delta(\Gamma_k)=:d > 0,$ contains a subsequence $\Gamma_{k_j}$ which converges weakly to some separated set $\Gamma'$ satisfying  $ \Delta(\Gamma')\geq d.$
In particular, every sequence of translates  $\Gamma -x_k: x_k\in\R,$ of a separated set $\Gamma$ contains a subsequence $\Gamma-x_{k_j}$ which  converges weakly to some separated set $\Gamma'$ satisfying $\Delta(\Gamma')\geq \Delta(\Gamma).$
\end{lemma}

\subsection{Stability of shifts and crystalline measures}

\begin{theorem}\label{t5}
    Assume $\Gamma\subset\R$ is a separated set and  $\Ff\in W\cap C(\R)$  satisfies  \eqref{locfin}.
    Then $\Gamma$-shifts of $\Ff$ are $l^{\infty}$-stable if and only if no set $\Gamma^\ast\in W(\Gamma)$ supports a sparse crystalline measure with bounded coefficients whose spectrum lies in ${\rm Zer}(\hat\Ff).$
\end{theorem}

Recall that by Theorem \ref{t0}, under the conditions of this theorem, the $l^\infty$-stability of $\Gamma$-shifts implies the $l^p$-stability, for every $1\leq p<\infty.$

A stronger result is true for the generators $\Ff$ such that the set  ${\rm Zer}(\hat\Ff)$   
is a {\it finite union} of separated sets. 
It is proved in \cite[Theorem~2.2]{MR3667579} that a sparse crystalline measure with such spectrum is a generalized Poisson comb\footnote{ We also refer the reader to the paper \cite{MR3338010}, where the same statement is proved for separated sets~$\Gamma$.}. This and  Theorem \ref{t5} imply
\begin{corollary}\label{t6}
    Assume $\Gamma\subset\R$ is a separated set and  $\Ff\in W\cap C(\R)$   is such that the set $\Zer(\hat\Ff)$ is a finite union of separated sets.
    Then $\Gamma$-shifts of $\Ff$ are $l^{\infty}$-stable if and only if  no set $\Gamma^\ast\in W(\Gamma)$  supports a generalized Poisson comb whose spectrum lies in $\Zer\,(\hat\Ff)$.
\end{corollary}

 This observation establishes the connection between the sets $\Gamma$ and $\Zer(\hat{\Ff})$. Moreover, if one of these sets has "particularly nice" properties then the stability of shifts is "universal", i.e., it does not depend on the properties of the other set. More precisely, Corollary~\ref{t6} leads to the following results:
\begin{itemize} 
    \item  We give new examples of generators $\Ff$ such that $\Gamma$-shifts of $\Ff$ are $l^{\infty}$-stable for every separated set $\Gamma$. 
    \item  There exist separated sets $\Gamma=\{...<\gamma_{k-1}<\gamma_k<\gamma_{k+1}<..., k\in\Z\}$ satisfying 
    \begin{equation}\label{reldense}
        \gamma_{k+1}-\gamma_k\leq d,\quad  \mbox{for some } d>0,
    \end{equation}
    such that $\Gamma$-shifts of $\Ff$ are $l^{\infty}$-stable, for every generator $\Ff$ satisfying the assumptions imposed in Corollary~\ref{t6}.
\end{itemize}

The support $\Gamma$ of a generalized Poisson comb defined in \eqref{poi} is a subset of a separated periodic set, i.e. $\Gamma\subseteq a\Z+F,$ for some $a>0$ and a finite set $F$. Its spectrum (the support of its Fourier transform) is a subset of $(1/a)\Z+\tilde F,$ where $\tilde F$ is also a finite set. 
Moreover, one may check that the support of a generalized Poisson comb
meets some arithmetic progression in $\R$ in an infinite number of points. The same is true for its spectrum. 
In this connection, observe that there exist sparse crystalline measures (even with integer coefficients) such that the intersection of their support with any arithmetic progression is either empty or finite, see \cite{MR4129870,MR4206541}. 

We see that if $\Ff$ satisfies the assumptions of Corollary \ref{t6} and the intersection of $\Zer(\hat{\Ff})$ with every arithmetic progression is either finite or empty, then $\Gamma$-shifts of $\Ff$ are $l^{\infty}$-stable for every separated set $\Gamma$.

We also see that $\Gamma$-shifts of $\Ff$ are stable for every $\Ff$ satisfying the assumptions of  Corollary~\ref{t6}, provided $\Gamma$ has the property that no weak limit of its translates contains an infinite subset of an arithmetic progression.

\begin{claim}
  There is a separated set $\Gamma=\{\gamma_k, k\in\Z
\}$ satisfying \eqref{reldense}  such that no set $\Gamma^\ast\in W(\Gamma)$ contains an infinite subset of an arithmetic progression.  
\end{claim}

 \begin{proof}
      We use a slightly modified example from \cite{MR4206541}. One may also use the example in ~\cite{MR4129870}.

Given any numbers $a, b\in\R$, set $$h_{a,b}(z):=\sin(\pi z+a)-(1/2) \sin(z+b). $$ Set  $\Gamma(a,b):=$Zer$(h_{a,b})$.  Similarly to \cite{MR4206541}, one can check that for every $a,b$ we have:

(i) $\Gamma(a,b)$ is a separated subset of $\R$;

(ii)  $\Gamma(a,b)$ does not contain any arithmetic progression in $\R$.

We will use the lemma from \cite[Section~4]{MR379387}, which follows from the Skolem-Mahler-Lech theorem. It claims that if a trigonometric polynomial $h$ vanishes on an infinite subset of arithmetic progression $\alpha \Z + \beta, \alpha,\beta\in \R,$ then there exist $r,d \in \R$ such that $h$ vanishes on the full progression $r\Z + d$. 
Together with condition (ii), this implies 

(iii) $\Gamma(a,b)$ meets every arithmetic progression in (at most) finite number of points.

Since $\sin(\pi z+b)$ is a $2$-periodic function, it is clear that  every sequence of real numbers $t_k$ contains a subsequence $s_{j}$ such that the sequence $\sin (\pi (x-s_{j})+a)$ converges (uniformly on compacts) to  $\sin (\pi x+ a_1),$ for some number $ a_1\in\R.$ Similarly, one may find a subsequence $d_{l}$ of $s_{j}$ such that $\sin((x-d_l)+b)$ converges to $\sin(x+b_1).$ It follows that every element of the set  $W(\Gamma(a,b))$ is of the form $\Gamma(a_1,b_1)$. Now, conditions (i) and (iii) above prove the claim. \end{proof}

\section{Stability of shifts and interpolation: Proofs}\label{sec_proofs1}

\subsection{}\label{sect_lemma_above_proof}
Given a countable set $\Lambda\subset\R$, denote by 
\begin{equation}\label{lud}D^-(\Lambda):=\lim_{r\to\infty}\inf_{x\in\R}\frac{\#(\Lambda\cap (x,x+r))}{r}\end{equation}the lower uniform density of $\Lambda$.

We start with a result which extends the right-hand-side inequality in \eqref{eqstab} to more general sets $\Gamma$:

\begin{proposition}\label{lemaa_int_rhs_equiv_general}
Let $\Gamma \subset \bigcup\limits_{i=1}^N (\frac{1}{\alpha_i} \Z + \beta_i), \alpha_i >0, \beta_i \in \R, N\in\N,$ satisfy  the condition
\begin{equation}\label{dGammaZi}
    D^{-}\left(\Gamma \cap \left(\frac{1}{\alpha_i}\Z + \beta_i\right)\right) > 0, \quad \text{for every  }i =1,...,N.
\end{equation}
 Then the following statements are equivalent{\rm:}
\begin{enumerate}
    \item[$(i)$] There exists a positive $C$ such that
    \begin{equation}\label{sever_progressions_Ghat}
        \left\| \sum\limits_{k \in \Z} |\hat{\Ff}(t + \alpha_i k)|^2 \right\|_{L^{\infty}(0,{\alpha_i})} \le C,\quad i =1,...,N;
    \end{equation}
    \item[$(ii)$] There exists $C>0$ such that 
    \begin{equation*}
        \left\|\sum_{\gamma \in \Gamma} c_{\gamma} \Ff(\,\cdot - \gamma) \right\|^2_2 \le C \|\cc\|_2^2,\quad \mbox{for every $\cc=\{c_\gamma\} \in l^2(\Gamma).$}
    \end{equation*}
\end{enumerate}
\end{proposition}

We will need several lemmas.
\begin{lemma}[Bessel's inequality (see, e.g., \cite{MR1836633}, Chapter~4, Section~3)]\label{bes}
Let $\Gamma$ be a separated set, $a\in \R$. Then
\begin{equation}\label{equ2}
    \int_a^{a+1}\left| \sum_{\gamma\in\Gamma}c_\gamma e^{2\pi i \gamma t}  \right|^2\,dt\leq K \|{\bf c}\|_2^2,\quad \mbox{for every } {\bf c}=\{c_\gamma\}\in l^2(\Gamma),
\end{equation}where the constant $K>0$  depends only on $\Delta(\Gamma)$.   
\end{lemma}

Observe  (see, e.g.,  \cite[Lecture~2.5]{ou}) that \eqref{equ2} is equivalent to the inequality
\begin{equation}\label{equ20}
 \sum_{\gamma\in\Gamma}|f(\gamma)|^2\leq K\|f\|_2^2,  \quad \mbox{for every } f\in PW_{[a,a+1]}.
\end{equation}

\begin{lemma}\label{lem_s1}
    Assume $\Gamma\subset\Z$ is  a separated set satisfying $D^-(\Gamma)>0.$ Let $I\subset (0,1)$ be an interval satisfying $|I|<D^-(\Gamma)$. Then there exists  $K$ such that every function $h\in L^2(I)$
    admits a representation
    $$h(t)=\sum_{n\in\Gamma}a_ne^{2\pi i n t},\quad  \sum_{n\in\Gamma}|a_n|^2\leq K\|h\|_2^2.$$
\end{lemma}

This lemma follows from the fact that the exponential system $\{e^{2\pi i n t}:n\in\Gamma\}$ forms a frame in $L^2(I)$, see, e.g., \cite[Lecture~3]{ou}.

\begin{proof}[Proof of Proposition \ref{dGammaZi}]
For simplicity of presentation, we give a proof for the case $\Gamma \subset \Z \cup (\sqrt{2} \Z)$. Note that we do not require  $\Gamma$ to be separated.
Set $\Gamma_1:= \Gamma \cap \Z, \Gamma_2:= \Gamma \cap (\sqrt{2}\Z), $ and 
$$
f(x) := \sum_{\gamma \in \Gamma} c_{\gamma} \Ff(x - \gamma) = \sum_{n \in \Gamma_1} a_{n} \Ff(x - n) + \sum_{m \in \Gamma_2} b_{m} \Ff(x -m)=:f_1(x)+f_2(x),$$ where ${\bf a}=\{a_n\} \in l^2(\Z), {\bf b}=\{b_m\}\in l^2(\sqrt2\,\Z). 
$

$(i) \rightarrow (ii)$.
Observe that $\sum\limits_{n \in \Z} a_n e^{2 \pi i n t}$ is a $1$-periodic function. Therefore, 
passing to the Fourier transform, we have $$
\|f_1\|^2=\|\hat f_1\|_2^2=  \int\limits_{\R} \left|\sum\limits_{n \in \Z} a_{n} e^{2\pi i n t}\right|^2 |\hat{\Ff}(t)|^2 \, dt =$$ \begin{equation}\label{eqst}\int\limits_{0}^1 \left|\sum\limits_{n \in \Z} a_{n} e^{2\pi i n t}\right|^2 \sum_{k \in \Z}|\hat{\Ff}(t + k)|^2 \, dt \le  \|{\bf a}\|_2^2 \ess_{t\in(0,1)}\sum_{k\in\Z}|\hat\Ff(t+k)|^2. \end{equation}
A similar estimate holds for the second sum. 
This and \eqref{sever_progressions_Ghat} give
$$
\|f\|^2_2 \leq 2( \|\hat{f_1}\|^2_2+\|\hat{f_2}\|^2_2)  \le C\|\cc\|^2_2.
$$

$(ii) \rightarrow (i)$. We assume that condition $(i)$ is not true, and prove that $(ii)$ is not true either.

Fix a number $d$ satisfying $0<d<D^-(\Gamma\cap\Z)$.
By our assumption, we may assume that  for every $N$ there exist $\varepsilon>0$ and a set $S\subset (0,1),|S|=\varepsilon,$ such that $$ \sum\limits_{k \in \Z} \left|\hat{\Ff}\left(t + k \right)\right|^2  \ge N, \quad t\in S.$$ Moreover, we may assume that $S$ is a subset of an interval $I\subset(0,1)$ of length $d$. 

By Lemma \ref{lem_s1}, we may find coefficients ${\bf a}=\{a_n\}\in l^2(\Gamma\cap\Z)$ such that $\|{\bf a}\|_2\leq K$ and
$$\sum_{n\in\Gamma}a_ne^{2\pi i n t}=\frac{1}{\sqrt\varepsilon}{\bf 1}_S(t),\quad t\in I.$$
Put
$$f(x):=\sum_{n\in\Gamma\cap\Z}a_n \Ff(x-n).$$Then
$$\|f\|_2^2=\|\hat f\|_2^2=\int_0^1 \left|\sum_{n\in\Gamma\cap\Z}a_ne^{2\pi i n t}\right|^2\sum_{k\in\Z}\left|\hat\Ff(t+k)\right|^2\,dt\geq $$
$$\int_S \left|\sum_{n\in\Gamma\cap\Z}a_ne^{2\pi i n t}\right|^2\sum_{k\in\Z}\left|\hat\Ff(t+k)\right|^2\,dt\ge N.$$

We conclude that the constant $C$ in  $(ii)$ cannot be less than $N/K$.
\end{proof}

\subsection{Proof of Lemma \ref{Integer_rhs_cond}} The right-hand-side inequality in 
\eqref{eqstab}  is proved in \eqref{eqst}. The sharpness of this inequality can be proved similarly to the second part of the proof above, where instead of Lemma \ref{lem_s1} one uses the fact that the exponentials $\{e^{2\pi i n t}\}$ form an orthonormal basis for $L^2(0,1)$. The proof of the left-hand-side inequality in \eqref{eqstab} and its sharpness is similar.

\subsection{Proof of Lemma ~\ref{lemma_above_est}}
We use Lemma \ref{bes} and  pass  to the Fourier transform:
$$
 \left\|\sum_{\gamma \in \Gamma} c_{\gamma} \Ff(\,\cdot - \gamma) \right\|^2_2 = \sum_{k \in  \Z} \int\limits_{k}^{k+1} \left|\sum\limits_{\Gamma} c_{\gamma} e^{2\pi i \gamma t}\right|^2 |\hat{\Ff}(t)|^2 \, dt  \le K \| |\hat{\Ff}|^2\|_W \|\cc\|^2_2.
$$

\subsection{Proof of Theorem \ref{t3}} 

 We will need several auxiliary lemmas.
\begin{lemma}\label{l1}
    Assume $\varphi\in W$. Then
    $$\liminf_{N\to\infty}N^{-1}\#\{k\in\N: \|\varphi\|_{L^\infty(k,k+1)}>1/N\}=0.$$
\end{lemma}

\begin{proof}  The proof is standard. Set 
$$a_n:=\|\varphi\|_{L^\infty(n-1,n)}+\|\varphi\|_{L^\infty(-n,-n+1)},\quad n=1,2,...$$Then $\{a_n\}$ belongs to $l^1$, and the same is true for its   non-increasing rearrangement  $\{a_n^\ast\}_{n=1}^\infty$.
Therefore, for every $\epsilon>0$ there are infinitely many numbers $k$ such that $a_k^\ast<\epsilon/k$, which proves the lemma. \end{proof}

We also need the following well-known

\begin{lemma}[see, e.g., \cite{ou}, Proposition~4.6]\label{lint}
  A separated set  $\Gamma$ is an interpolation set for $PW_S$ if and only if there exists  $K>0$ such that\begin{equation}\label{sta}\int\limits_{S}\left|\sum_{\gamma\in\Gamma}c_\gamma e^{2\pi i \gamma t}\right|^2\,dt\geq K\|\cc\|^2_{l^2},\quad  \mbox{for every } \cc=\{c_\gamma\}\in l^2(\Gamma). \end{equation}
\end{lemma}

    \begin{proof}[Proof of Theorem \ref{t3}]
        (i) Assume $\Gamma$ is  an interpolation set for $PW_{\{|\hat\Ff|>r\}}$, for some $ r>0.$ 
       Using \eqref{sta},  for every           $f(x)$ in \eqref{f0}  (with $p=2$)  we get
$$\|f\|_2^2=\|\hat f\|_2^2=\int\limits_{\R}\left|\sum_\Gamma c_\gamma e^{2\pi i \gamma t}\right|^2|\hat\Ff(t)|^2\,dt\geq r^2\int\limits_{\{|\hat\Ff|> r\}}\left|\sum_\Gamma c_\gamma e^{2\pi i \gamma t}\right|^2\,dt\geq K r^2 \|\cc\|_{2}^2.$$
Since $\Ff$ satisfies condition~\eqref{e1} from Lemma~\ref{lemma_above_est}, we see that $\|f\|_2 \le C \|\cc\|_2$, and the $l^2$-stability of $\Gamma$-shifts follows. 

(ii)  
Assume that $\Gamma$ is not an interpolation set for $PW_{\{|\hat\Ff|>r\}},$ for every $r>0$. We will check that $\Gamma$-shifts of $\Ff$ are not $l^2$-stable. 

By Lemma \ref{l1},
there is a sequence  ${\bf \epsilon}=\epsilon_k\to 0$  such that for every $\epsilon$ there is an integer $N$ such that $|\hat\Ff(t)|^2<\epsilon/N$ outside  a set $S_N$ which consists of  at most $ N$ intervals $[k,k+1)$. Since $\Gamma$ is not an interpolation set  for $PW_{\{|\hat\Ff|^2>\epsilon/N\}}$, by Lemma \ref{lint}, there are coefficients $\cc=\{c_\gamma\}$ satisfying $\|\cc\|_{2}=1$ and 
$$\int\limits_{\{|\hat\Ff|^2>\epsilon/N\}}\left|\sum_\Gamma c_\gamma e^{2\pi i \gamma t}\right|^2|\hat\Ff(t)|^2\,dt<\|\hat\Ff\|_\infty^2\int\limits_{\{|\hat\Ff|^2>\epsilon/N\}}\left|\sum_\Gamma c_\gamma e^{2\pi i \gamma t}\right|^2\,dt<\epsilon.$$

Let  $f$ be any function defined in (\ref{f0}). Then
$$\|f\|_2^2=\int_\R \left|\sum_\Gamma c_\gamma e^{2\pi i \gamma t}\right|^2|\hat\Ff(t)|^2\,dt=\int_{\{|\hat\Ff|^2>\epsilon/N\}}+\int_{S_N\setminus \{|\hat\Ff|^2>\epsilon/N\}}+\int_{\R\setminus S_N}.$$ The first integral above is less than $\epsilon.$ By (\ref{equ2}), the second one satisfies
$$\leq \frac{\epsilon}{N}\int_{S_N}\left|\sum_\Gamma c_\gamma e^{2\pi i \gamma t}\right|^2\,dt=\frac{\epsilon}{N}\sum_{k\in S_N}\int_{k}^{k+1}\left|\sum_\Gamma c_\gamma e^{2\pi i \gamma t}\right|^2\,dt\leq K\epsilon.$$  By \eqref{equ2}, 
$$\int_{\R\setminus S_N}\leq K \sum_{k\not\in S_N}\||\hat\Ff|^2\|_{L^{\infty}(k,k+1)}.$$
Since $|\hat\Ff|^2\in W$
the last integral tends to zero as $\epsilon\to0.$       Therefore, we may choose coefficients ${\bf c}=\{c_\gamma\}$,
 $\|{\bf c}\|_2=1$, so that the $L^2$-norm of the corresponding function $f$ will be arbitrarily small.
\end{proof}

\subsection{Proof of Theorem \ref{t1}}

We will need an important result from \cite{MR3336091}:
\begin{lemma}\label{spec_invar}
  Let $\Lambda$ and $\Gamma$ be separated sets in $\R$. Let $A \in \Cc^{\Lambda \times \Gamma}$ be a matrix such that
  \begin{equation}\label{conv_domin}
        |A_{\lambda, \gamma}| \le \theta(\lambda - \gamma) \quad \lambda \in \Lambda, \gamma \in \Gamma \quad \text{for some  } \theta \in W.
  \end{equation}
  Assume that there exist a $p_0 \in [1,\infty]$ and $C_0 > 0,$ such that
\begin{equation}\label{e00}\|A\cc\|_{p_0} \geq  C_0\|\cc\|_{p_0}\quad \mbox{for all } \cc\in l^{p_0}(\Gamma).\end{equation}
Then there exists a constant $C > 0$ independent of $p$ such that, for all $p \in [1,\infty]$,
\begin{equation}\label{e01}\|A\cc\|_p \geq C\|\cc\|_p,\quad \mbox{for all } \cc\in l^p(\Gamma).\end{equation}
\end{lemma}

\begin{proof}
    
We divide the proof into steps.

{\bf Step~1.} Denote by $T_{\gamma}$ the translation operator $T_{\gamma}\Ff (\cdot)= \Ff( \cdot - \gamma)$.  We introduce the operator $D$ acting on the space of sequences  by the formula
\begin{equation}
    D \, \cc = D\{c_{\gamma}\}_{\gamma \in \Gamma} := \sum_{\gamma \in \Gamma} c_{\gamma} T_{\gamma} \Ff.
\end{equation}
It is easy to check (see e.g., \cite[Lemma~8.1]{uz}) that $D$ maps $l^p(\Gamma)$ to $L^p(\R)$ for every $1 \le p \le \infty.$

{\bf Step~2.} Next, we consider the operator $\mathcal{C}$ that is dual to $D$. Note that
$$
\mathcal{C}f = \left( \langle f, T_{\gamma } \Ff\rangle \right)_{\gamma \in \Gamma}
$$
and $\mathcal{C}$ is a bounded operator acting from $L^p(\R)$ to $l^p(\Gamma)$ for every $1 \le p \le \infty.$

{\bf Step~3.} The composition operator $\mathcal{C}D$ is given by
$$
\mathcal{C} D \, \cc = \mathcal{C}D \, \{c_{\gamma}\}_{\gamma \in \Gamma} = \left( \sum\limits_{
\gamma \in \Gamma} c_{\gamma} \langle T_{\gamma} \Ff, T_{\gamma'} \Ff \rangle \right)_{\gamma' \in \Gamma} =  G \, \cc,
$$
where by $G$ we denoted the Gramian matrix with entries $G_{\gamma, \gamma'} = \langle T_{\gamma}\Ff, T_{\gamma'}\Ff \rangle$.
Note that $\mathcal{C} D$ is also bounded on $l^p(\Gamma)$ for every  $1 \le p \le \infty.$
Moreover, since
\begin{equation}
G_{\gamma, \gamma'} = (\Ff \ast \tilde{\Ff}) (\gamma - \gamma'), \quad \text{ where }  \tilde\Ff(x)=\bar{\Ff}(-x),
\end{equation}
the matrix $G$ satisfies condition~\eqref{conv_domin}.

{\bf Step~4.} Since $\Gamma$ shifts of $\Ff$ are $l^2$-stable, we know that $\{ T_{\gamma}\Ff : \gamma \in \Gamma \}$ is a Riesz basis for $V^2_{\Gamma}(\Ff)$,  therefore its Gramian is invertible on $l^2(\Gamma),$ i.e. $\|G \cc\|_2 \ge C\|\cc\|_2.$

Hence, the assumptions of Lemma~\ref{spec_invar} are satisfied, and we deduce that $G$ is invertible simultaneously for all $p \in [1, \infty].$

Combining all together, we arrive at
$$
C_1 \|\cc\|_p \le \|G \cc\|_p = \|\mathcal{C}D \cc\|_p \le C_2 \|D \cc\|_p =C_2 \left\|\sum\limits_{\gamma} c_{\gamma} \Ff(\cdot - \gamma)\right\|_p \le C_3 \|\cc\|_p, 
$$
where by $C_1, C_2, C_3$ we denoted some positive constants.
Thus, $\Gamma$-shifts of $\Ff$ are $l^p$-stable for every $p \in [1, \infty].$
\end{proof}

\section{Stability of shifts and crystalline measures: Proofs}
We say that 
 $\Gamma$-shifts of $\Ff$ are  $c_0$-stable, if condition \eqref{stab} with $p= \infty$ holds true for  every sequence ${\bf c}\in c_0(\Gamma).$ Similarly, 
 $\Gamma$-shifts of $\Ff$ are  $c_{00}$-stable, if condition \eqref{stab} with $p =\infty$ holds true for  every {\it finite} sequence ${\bf c}.$ 

 Theorem \ref{t5} follow easily from
\begin{theorem}\label{schwartz_approx_theorem}
Assume that $\Gamma$ is a separated set, $\Ff\in W\cap C(\R)$ and satisfies \eqref{locfin}.
   The following statements are equivalent:
\begin{enumerate}
    \item[$(i)$] $\Gamma$-shifts of $\Ff$ are not $l^\infty$-stable{\rm;}
    \item[$(ii)$] $\Gamma$-shifts of $\Ff$ are not $c_0$-stable{\rm;}
    \item[$(iii)$] $\Gamma$-shifts of $\Ff$ are not $c_{00}$-stable{\rm;}
    \item[$(iv)$] There is a set $\Gamma^\ast\in W(\Gamma)$ such that $\Gamma^\ast$-translates of $\Ff$ are not $l^\infty$-independent, i.e. there are non-trivial coefficients $\{c_{\gamma^\ast}\}\in l^\infty(\Gamma^{\ast})$ such that
     \begin{equation}\label{cF_v}
         \sum_{\gamma^\ast\in\Gamma^\ast}c_{\gamma^\ast}\Ff(x-\gamma^\ast)\equiv0;
     \end{equation}
    \item[$(v)$] There is a set $\Gamma^\ast\in W(\Gamma)$ and a sequence ${\bf c}=\{c_{\gamma^\ast}\}\in l^\infty(\Gamma^\ast)$ such that  measure $$\mu=\sum_{\gamma^\ast\in\Gamma^\ast}c_{\gamma^\ast}\delta_{\gamma^\ast}$$ is a sparse crystalline measure whose spectrum lies in ${\rm Zer}(\hat\Ff)$.
\end{enumerate}
\end{theorem}

\begin{remark}
    In the assumptions of Theorem~\ref{schwartz_approx_theorem}, one can prove that the condition
\begin{enumerate}
    \item[$(vi)$] $\Gamma$-shifts of $\Ff$ are not $l^{p}$-stable for $1\le p < \infty$,
\end{enumerate}
    is also equivalent to $(i)-(v)$. The  implication $(vi) \rightarrow (i)$ follows from  Theorem~\ref{t0}.
    We omit the proof of the reverse implication.   
   Therefore,  if  $\Gamma$-shifts are stable for some  $1\leq p\leq\infty,$ then they are stable for every $p$.
\end{remark}

\begin{proof} We divide the proof into steps.

\noindent {\bf Step 1.} The implications  $(iii)\rightarrow (ii)\rightarrow (i)$ are obvious.

\noindent {\bf Step 2.} $(i)\rightarrow (ii)$. We assume that $\Gamma$-shifts are $c_0$-stable and prove that they are $l^\infty$-stable.

We will need a well-known stability property for surjective operators: 
     Assume $X, Y$ are Banach spaces, and the bounded linear operator $A: X \to Y$ is a surjection. Then there exists $\epsilon=\epsilon(A) > 0$ such that for
every linear operator $B: X \to Y$ with $\|B\| < \epsilon,$  the operator $A + B: X \to Y$ is a surjection.

 Let $T$ be an operator defined on the space $c_0(\Gamma)$  by $$ T{\bf c}=\sum_{\gamma \in \Gamma}c_\gamma\Ff(x-\gamma),\quad {\bf c}=\{c_\gamma\}\in c_0(\Gamma).$$It is clear that $T$ is a bounded operator from $c_0(\Gamma)$ to  $C_0(\R)$,  and its dual is the operator$$T^\ast: M(\R)\to l^1(\Gamma),\quad T^{\ast} \mu=\left\{\int_\R \Ff(x-\gamma)\,d\mu(x),\,\gamma\in\Gamma\right\}=\check\Ff\ast\mu|_\Gamma,$$where $\check\Ff(x)=\Ff(-x).$
 Then $c_0$-stability is equivalent to the property that there exists $K>0$ such that  $\|T{\bf c}\|_\infty\geq K\|{\bf c}\|_\infty,$ for every ${\bf c}= \{c_{\gamma}\}\in c_0(\Gamma),$ which in turn is equivalent to the property 
     that the operator $T^\ast$ is onto. 

 Let $0<\delta<1$ and let  $h$ be a non-negative smooth function supported by $[-\delta,\delta]$  satisfying $\int\limits_{\R} h(x)\,dx=1$. Set
     $$A: M(\R)\to l^1(\Gamma),\quad A \mu=\int\limits_\R\int\limits_{-\delta}^\delta \Ff(x-\gamma-s)h(s)\,ds\,d\mu(x)|_{\Gamma}=\check\Ff\ast h\ast\mu|_\Gamma.$$
    
  Fix  any small number $\epsilon>0$ and consider the sum
     $$\sum_{\gamma\in\Gamma} |\Ff(x-\gamma)-(\Ff\ast h)(x-\gamma)|=\sum_{\gamma\in\Gamma}\left|\int_{-\delta}^\delta (\Ff(x-\gamma)-\Ff(x-s-\gamma))h(s)\,ds\right|.$$ 
     Since $\Ff\in W$, there is a large number $R=R(\epsilon)$ such that   
     for every $x\in\R,$
     $$\sum_{\gamma\in\Gamma, |x-\gamma|>R}(|\Ff(x-\gamma)|+|\Ff(x-s-\gamma)|)\leq \epsilon,\quad |s|\leq\delta<1.$$
Further, since $\Gamma$ is separated, there is an integer $N=N(R)$ such that
$$\#\{\gamma\in\Gamma:|\gamma-x|\leq R\}\leq N,\quad \mbox{for every } x\in\R.$$
Now, condition $\Ff\in C(\R)$ implies that there exists  $\delta_0$ such that $$|\Ff(t)-\Ff(t-s)|\leq \epsilon/N,\quad |t|\leq R, \ |s|<\delta<\delta_0.$$
This gives
\begin{equation}\label{dif_G_gamma}
\sum_{\gamma\in\Gamma} |\Ff(x-\gamma)-\Ff(x-\gamma-s)|=\sum_{|\gamma-x|>R}+\sum_{|\gamma-x|\leq R}\leq 2\epsilon,\quad |s|\leq \delta_0.    
\end{equation}
We see that 
$$\|(A-T^\ast)\mu\|_{l^1(\Gamma)}=\sum_{\gamma\in\Gamma}\left|\int\limits_\R\int\limits_{-\delta}^\delta\left(\Ff(x-\gamma)-\Ff(x-\gamma-s)\right)h(s)ds\,d\mu(x)\right|\leq $$$$\|\mu\|\int\limits_{-\delta}^\delta\sum_{\gamma\in\Gamma} \left|\Ff(x-\gamma)-\Ff(x-\gamma+\tau)\right| h(s)ds\leq 2\epsilon\|\mu\|. $$
   We conclude that the operator $A$ is onto. 
     
     Observe that $h\ast\mu\in L^1(\R)$, for every $\mu\in M(\R)$. This means that the operator
     $$B:L^1(\R)\to l^1(\Gamma),\quad B \varphi=\check\Ff\ast\varphi|_\Gamma,\quad \varphi\in L^1(\R),$$ is also onto. Therefore, its dual operator 
     $$B^\ast: l^\infty(\Gamma)\to L^\infty(\R),\quad B^\ast {\bf c}=\sum_\Gamma c_\gamma\Ff(x-\gamma),$$ satisfies $\|B^\ast {\bf c}\|_\infty\geq C\|{\bf c}\|_\infty,$ which means that $\Gamma$-translates are $l^{\infty}$-stable.

{\noindent \bf Step 3.} $(ii)\rightarrow (iii)$. Assume that $\Gamma$-shifts of $\Ff$ are not $c_0$-stable, i.e. for every $\epsilon>0$ there is a sequence ${\bf c}\in c_0(\Gamma)$ such that $\|{\bf c}\|_\infty=1$ and $$\left\|\sum_{\gamma\in\Gamma}c_\gamma\Ff(x-\gamma)\right\|_\infty\leq\epsilon.$$ Choose $N$ so large that $|c_\gamma|<\epsilon$ for every $\gamma\in\Gamma, |\gamma|\geq N.$ Since $\Ff\in W$, we get
$$\left\|\sum_{\gamma\in\Gamma, |\gamma|\geq N}c_\gamma\Ff(x-\gamma)\right\|_\infty< \epsilon \|\Ff\|_W.$$ We conclude that
$$\left\|\sum_{\gamma\in\Gamma, |\gamma|<N}c_\gamma\Ff(x-\gamma)\right\|_\infty\leq \epsilon(1+\|\Ff\|_W),$$which proves property (iii).

{\noindent \bf Step 4.} 
$(i)\rightarrow (iv)$.  Write $\Gamma=\{\gamma_k:k\in\Z\}$.
    By (i), for every $n\in\N$ there is a  sequence  ${\bf c}^{(n)}=
    \{c^{(n)}_{k}\}_{k\in\Z}\in l^\infty(\Z)$ such that $\|{\bf c}^{(n)}\|_\infty=1$ and 
    \begin{equation}\label{ee}
    \left\|\sum\limits_{k\in\Z} c^{(n)}_{k} \Ff(x - \gamma_k)\right\|_{\infty} < 1 /n.
    \end{equation}
    
    For every $n$, choose an integer $j(n)\in\Z$ such that $|c^{(n)}_{j(n)}|>1-1/n$.
    Clearly,  $0\in \Gamma -\gamma_{j(n)}.$     
    By Lemma \ref{lsep}, passing to a subsequence, we may assume that
$\Gamma - \gamma_{j(n)}$ converges weakly to a non-empty set $ \Gamma'=\{\gamma'_k:k\in\Z\} \in W(\Gamma)$. Passing again to a subsequence, we find  a sequence $\{c_k\}\in l^\infty(\Z)$ such that   $$c_{k+j(n)}^{(n)}\to c_k,\quad n\to\infty, \ k\in\Z.
$$Clearly,  $|c_0|=1,$ so that the limit sequence is non-trivial.
Since $\Ff\in W\cap C(\R)$, it follows from \eqref{ee} that 
$$ \sum\limits_{k\in\Z} c_{k} \Ff(x - \gamma_k')\equiv 0.$$ 

{\bf \noindent Step 5.} $(iv)\rightarrow (iii)$.
Write $$\Gamma=\{...<\gamma_{n-1}<\gamma_n<\gamma_{n+1}<...\} \quad  \text{and} \quad  
\Gamma^\ast=\{...<\gamma^\ast_{n-1}<\gamma^\ast_n<\gamma^\ast_{n+1}<...\}.$$ 

Clearly,  $(iv)$ implies that $\Gamma^\ast$-translates of $\Ff$ are not $l^\infty$-stable. By Steps 2 and 3, 
they are not $c_{00}$-stable. Therefore, for every $\epsilon>0$ there exists a finite sequence ${\bf c}=\{c_n\}$ such that
$$\|\cc\|_\infty=1,\quad \left\|\sum_{n}c_{n}\Ff(x-\gamma^\ast_n)\right\|_\infty\leq\epsilon.$$

Since $\Gamma^\ast\in W(\Gamma)$, for every $N\in\N$ and $\delta>0$ there exists $j\in \Z$ such that
$$|\gamma_{n+j}-\gamma^\ast_n|\leq\delta,\quad |n|\leq N.$$ Assumption $\Ff\in W(\R)\cap C(\R)$ obviously
implies that
$$\left\|\sum_{n}c_n\Ff(x-\gamma_{n+j})\right\|_\infty<2\epsilon,$$provided  $\delta$ sufficiently small. This proves  Step 5.

{\bf \noindent Step 6.} $(v)\rightarrow (iv)$. This follows from an approximation argument.
Choose any Schwartz function $H$ such that \begin{equation}\label{eh}\hat H(0)=1, \ H(t), \hat H(t)\geq 0,\ t\in\R, \ \hat H(t)=0, \ |t|\geq 1, \end{equation}and put $H_a(t):=H(at), a>0.$ Then $\widehat{H_a}(t)=\hat H(t/a)/a.$

Let $\epsilon, \delta$ be positive numbers, and set $\Ff_{\epsilon,\delta}:=(1/\epsilon)(\Ff\ast H_{1/\epsilon})H_\delta $. Then $\Ff_{\epsilon,\delta}\in S(\R)$ and $\widehat{\Ff_{\epsilon,\delta}}=(1/\delta)(\hat G\hat H_\epsilon)\ast \hat H_{1/\delta}.$ Since $\hat H_\epsilon(t)=0, |t|\geq 1/\epsilon$ and 
$\hat H_{1/\delta}(t)=0, |t|\geq\delta$, it follows that $\widehat{\Ff_{\epsilon,\delta}}(t)=0$
for $|t|\geq \delta+1/\epsilon.$
We now apply the generalized Poisson formula \eqref{crys} to the function $\Ff_{\epsilon,\delta}(x-\cdot\,)$:
$$\sum_{\gamma^\ast\in \Gamma^\ast}c_{\gamma^\ast}\Ff_{\epsilon,\delta}(x-\gamma^\ast)=
\sum_{s\in \Zer(\hat\Ff), |s|\leq 1/\epsilon+\delta}a_s e^{2\pi i sx}\left((\hat \Ff\hat H_{\epsilon})\ast (1/\delta)\hat H_{1/\delta}\right)(s).$$
Using \eqref{eh}, we get 
$$\left|\sum_{\gamma^\ast\in \Gamma^\ast}c_{\gamma^\ast}\Ff_{\epsilon,\delta}(x-\gamma^\ast)\right|\leq H(0) \max_{|s|<1/\epsilon+\delta}|a_s|\max\limits_{\substack{s\in \Zer(\hat\Ff) \\ |s|\leq 1/\epsilon+\delta}}\max_{|t-s|\leq\delta}
|(\hat \Ff\hat H_{\epsilon})(t)|,\quad x\in\R.$$

Since $\Ff\in W,$ it is easy to check that $\Ff\ast H_{1/\epsilon}\in W.$
Since $\hat\Ff\in C(\R)$ and  $\hat\Ff$ vanishes on $\Zer(\hat\Ff)$, we see that the right-hand side of the above inequality tends to zero as $\delta\to 0,$ 
so that
\begin{equation}\label{ein}
\left\| \sum_{\gamma^\ast\in \Gamma^\ast}c_{\gamma^\ast}\Ff_{\epsilon,\delta}(x-\gamma^\ast) \right\|_\infty\to0,\quad \delta\to0.  
\end{equation}

Condition $\Ff\in W$ implies $\Ff\ast H_{1/\epsilon}\in W$. We now use  Lemma 8.1 in \cite{uz}, which gives
\begin{equation}\label{ein2}\left\| \sum_{\gamma^\ast\in \Gamma^\ast}c_{\gamma^\ast}\Ff\ast H_{1/\epsilon}(x-\gamma^\ast) \right\|_\infty\leq C(\Gamma^\ast)\|\Ff\ast H_{1/\epsilon}\|_W\|{\bf c^\ast}\|_\infty,\quad {\bf c^\ast}=\{c_{\gamma^\ast}\}\in l^\infty(\Gamma^\ast).\end{equation}

Choose any small number $\rho>0$.
Write 
$$\sum_{\gamma^\ast\in \Gamma^\ast}c_{\gamma^\ast}(\Ff\ast H_{1/\epsilon})(x-\gamma^\ast) (H(0)-H_\delta(x-\gamma^\ast))=\sum_{\gamma^\ast\in\Gamma^\ast, |\gamma^\ast-x|\leq R} + \sum_{\gamma^\ast\in\Gamma^\ast, |\gamma^\ast-x|> R}=:S_1+S_2.$$Since $\Ff\ast H_{1/\epsilon}\in W$, by \eqref{ein2}, we may choose $R=R(\epsilon)$ so large that $|S_2(x)|\leq\rho$, for every $x\in\R.$ Clearly, the functions $H_\delta(x)$ are uniformly bounded and converge uniformly on compacts in $\R$ to $H(0)\ne0$ as $\delta\to 0.$ This and \eqref{ein2} shows that  $|S_1|<\rho$ for every $x\in\R$, provided $\delta$ is small enough.
By \eqref{ein} we conclude that
\begin{equation}\label{eqve}\sum_{\gamma^\ast\in \Gamma^\ast}c_{\gamma^\ast}(\Ff\ast H_{1/\epsilon})(x-\gamma^\ast)\equiv 0.
\end{equation}

    Fix any small $\sigma > 0.$
    Similarly to the argument on Step~2 (see~\eqref{dif_G_gamma}), one can find $\delta_0$ such that 
  $$  \sum_{\gamma^\ast\in \Gamma^\ast}|\Ff(x-\gamma^\ast) - \Ff(x-s-\gamma^{\ast})| \le \sigma, \quad  \text{   for every } x\in \R, |s| < \delta_0. 
  $$
    For sufficiently small $\epsilon$, the inequality
    $$\int\limits_{\R\setminus(-\delta_0,\delta_0)} \frac{1}{\epsilon}  H_{1/\epsilon}(s) ds < \frac{\sigma}{\|\Ff\|_\infty}.
    $$
    holds. Since $\int_\R (1/\epsilon)H_\epsilon=1$, these two inequalities imply
        $$\left|\sum_{\gamma^\ast\in \Gamma^\ast}c_{\gamma^\ast}(\Ff(x-\gamma^\ast) - (\Ff\ast (1/\epsilon) H_{1/\epsilon})(x-\gamma^\ast)\right|=$$    
     $$\left|\sum_{\gamma^\ast\in \Gamma^\ast}c_{\gamma^\ast}\int_\R(\Ff(x-\gamma^\ast) - \Ff(x-\gamma^\ast-s)) (1/\epsilon)H_\epsilon(s)\,ds\right| \le2\sigma\|{\bf c^\ast}\|_\infty.$$This and \eqref{eqve} prove $(iv)$.

{ \bf   \noindent Step 7.} $(iv)\rightarrow (v)$. Condition $(iv)$ means that the measure $$\mu=\sum_{\gamma^\ast\in\Gamma^\ast}c_{\gamma^\ast}\delta_{\gamma^\ast}$$satisfies
    $$\mu\ast \Ff(x)\equiv 0.$$
    Since the coefficients $c_{\gamma^{\ast}}$ are bounded and 
   $\Gamma^{\ast}$ is separated, this measure is a tempered distribution. 

    \begin{claim}\label{c1}
        We have $\langle\mu,\hat\phi\rangle=0$, for every $\phi\in S(\R)$ with compact support satisfying $\supp(\phi)\cap \Zer(\hat\Ff)=\emptyset.$
    \end{claim}

\begin{proof}
    Since $\phi$ has compact support, it suffices to prove this claim for the case when
    $\phi$ vanishes outside a closed interval $I, I\cap \Zer(\hat\Ff)=\emptyset.$

    Let $\check\phi(x):=\hat\phi(-x)$ be  the inverse Fourier transform of $\phi$. 
    Then since$$\langle \mu,\hat\phi \rangle =\sum_{\gamma^\ast\in\Gamma^\ast}c_{\gamma^\ast}\hat\phi(\gamma^\ast)=\sum_{\gamma^\ast\in\Gamma^\ast}c_{\gamma^\ast} \check\phi(-\gamma^\ast)=(\mu\ast \check\phi)(0),$$it suffices to prove that $\mu\ast\check\phi\equiv0.$

    Since $\hat\Ff$ is continuous and does not vanish on $I,$ by the local form of the Wiener–L\'{e}vy Theorem  (see   \cite[Chapter~6]{MR1802924}) there is a function $h\in L^1(\R)$ such that  $\hat h(t)=1/\hat\Ff(t)$ on $I$.  
    Put $H:=h\ast\check\phi,$ then 
    $$
    \hat H=\hat h \phi=\left\{\begin{array}{cc}
     \phi/\hat\Ff    & \mbox{on } I \\
      0  &  \mbox{outside }  I. 
    \end{array}\right.
    $$
    Hence, $\hat\Ff\hat H=\phi,$ and so $\check\phi= \Ff\ast H.$ Since $\mu\ast\Ff\equiv 0,$ we conclude that $$\mu\ast \check\phi=\mu\ast \Ff\ast H\equiv0.$$
\end{proof}

\begin{claim}\label{c2}
    Let $I_0$ be an open interval and $t_0\in I_0$ such that $I_0\cap \Zer(\hat\Ff)=\{t_0\}.$
    Assume $\phi\in S(\R)$ vanishes outside $I_0$ and satisfies $\phi(t_0)=\phi'(t_0)=0.$
    Then $\langle\mu,\hat\phi\rangle=0.$
\end{claim}

\begin{proof}  
Without loss of generality, we may assume that $I_0=(-a,b)$, for some $0<a,b<1$
 and $t_0=0$. Fix a number $\epsilon, 0<\epsilon<\min\{a,b\}/4$. 
 
 Choose any functions $\psi_1,\psi_2\in S(\R)$ satisfying $$\psi_1(t)+\psi_2(t)=1, \quad |t|\leq 4, \quad \psi_1(t)=0,\ |t|\geq 2,\quad  \psi_2(t)=0,\ |t|\leq 1.$$
Set
$\phi_j(t):=\phi(t)\psi_j(t/\epsilon), j=1,2$. Then $$\phi_1+\phi_2=\phi,\quad \phi_1(t)=0, \ |t|\geq 2\epsilon,\ \  \phi_2(t)=0,\ |t|\leq \epsilon.$$

By Claim \ref{c1}, $\langle\mu,\hat\phi_2\rangle=0.$ Let us now prove that $\langle\mu,\hat\phi_1\rangle=0.$
Below we denote by $C$ and $C_j$ positive constants.

Since $\phi(0)=\phi'(0)=0,$ there is a constant $C=C(\phi)$ such that  $\|\phi_1''\|_\infty< C.$
   Using integration by parts, we get $$\hat\phi_1(x)=\int_{-2\epsilon}^{2\epsilon}e^{-2\pi i xt}\phi_1(t)\,dt=-\frac{1}{(2\pi x)^2}\int_{-2\epsilon}^{2\epsilon}e^{-2\pi i xt}\phi_1''(t)\,dt.$$  This  implies $|\hat\phi_1(x)|\leq C_1\epsilon/(1+|x|^2)$, and so $|\langle\mu,\hat\phi_1\rangle|<C_2\epsilon.$ Therefore,
$|\langle \mu,\hat\phi\rangle|=|\langle\mu,\hat\phi_1\rangle|<C_2\epsilon.$ Since $\epsilon$ can be chosen arbitrarily small, we conclude that $\langle\mu,\hat\phi\rangle=0$.
\end{proof}

\begin{claim}\label{c3}
    Let $I_0$ and  $t_0\in I_0$ satisfy the assumptions of Claim \ref{c2}. There is a constant $a_0$ such that the equation    \begin{equation}\label{eqe}\langle\mu,\hat\phi\rangle=a_0\phi(t_0)\end{equation} holds for every  $\phi\in S(\R)$ which vanishes outside $I_0$.
\end{claim}

\begin{proof}
We may assume that  $I_0=(-a,b)$ and $t_0=0$.
    Fix any functions  $\psi_j\in S(\R), j=1,2,$  which vanish outside $I_0$ and satisfy $\psi_1(0)=1,\psi'_1(0)=0, \psi_2(0)=0, \psi'_2(0)=1.$ Set
$$a_0:=\langle\mu,\hat\psi_1\rangle,\quad b_0:=\langle\mu,\hat\psi_2\rangle.$$For every function $\psi\in S(\R)$ which vanishes outside $I_0,$
the function $\phi(t):=\psi(t)-\psi(0)\psi_1(t)-\psi'(0)\psi_2(t)$ satisfies the assumptions of Claim \ref{c2}.
It follows that $\langle\mu,\hat\phi\rangle=0$ and so \begin{equation}\label{ab0}\langle \mu,\hat\psi \rangle =\psi(0)\langle\mu,\hat\psi_1\rangle+\psi'(0)\langle\mu,\hat\psi_2\rangle=a_0\psi(0)+b_0\psi'(0).\end{equation}It remains to prove that $b_0=0.$

Now, fix a small positive number $\epsilon$ and set  $\psi(t):=\psi_2(t/\epsilon)$. Then $\psi(0)=0,\psi'(0)=1/\epsilon$ and $\hat\psi(x)=\epsilon\hat\psi_2(\epsilon x).$ By \eqref{ab0},
$$|b_0|/\epsilon=|\langle\mu,\hat\psi\rangle|=\left|\epsilon\sum_{\gamma^\ast\in\Gamma^\ast}c_{\gamma^\ast}\hat\psi_2(\epsilon  \gamma^\ast)\right|\leq \|c_{\gamma^\ast}\|_\infty \epsilon\sum_{\gamma^\ast\in\Gamma^\ast}|\hat\psi_2(\epsilon  \gamma^\ast)|.$$

\begin{claim}
  There is a constant $C=C(\psi_2)$ such that   $$\epsilon\sum_{\gamma^\ast\in\Gamma^\ast}|\hat\psi_2(\epsilon  \gamma^\ast)|\leq C,\quad 0<\epsilon<1.$$
\end{claim}

This can be easily deduced from the assumptions that $\Gamma^\ast$ is separated and $\hat\psi_2\in S(\R)$.

Finally, we see that $|b_0|\leq C\epsilon,$ for some positive constant $C.$
 Since $\epsilon$ can be chosen arbitrarily small, we conclude that $b_0=0.$   The claim follows from \eqref{ab0}.
\end{proof}

We now complete the proof of Step 7.
Similarly to \eqref{eqe}, we have $\langle\mu,\hat\phi\rangle=a_k\phi(t_k)$, for every point $t_k\in \Zer(\hat\Ff)$ and every Schwartz function $\phi$ which vanishes outside a small interval around $t_k$ which does not intersect $\Zer(\hat\Ff).$ 

Choose any function $\phi\in S(\R)$ with compact support. For every $t_s\in\,\Zer\,(\hat\Ff)$, choose any numbers $0<\rho_s<\sigma_s$  such that $[t_s-\sigma_s,t_s+\sigma_s]\,\cap\,\Zer\,(\hat\Ff)=\emptyset$. Choose also a function $\phi_s\in S(\R)$ satisfying $\phi=\phi_s$ on $[t_s-\rho_s,t_s+\rho_s]$ and $\phi_s=0$ outside $(t_s-\sigma_s,t_s+\sigma_s).$ Since $\phi$ has compact support, only a finite number of these functions are non-trivial. By Claims \ref{c1} and \ref{c3},
$$\langle\mu,\hat\phi\rangle=\sum_{s\in \Zer(\hat\Ff)}\langle\mu,\hat\phi_s\rangle=\sum_{s\in \Zer(\hat\Ff)}a_s\phi(s)=\langle\hat\mu,\phi\rangle.$$ This means that $\hat\mu$ is an atomic measure concentrated on the set $\Zer\,\hat\Ff$.
\end{proof}

\section{Acknowledgements}
The authors are indebted to K.~Gr\"{o}chenig for providing a short proof of Theorem~\ref{t1} and other valuable suggestions. They also thank A.~Kulikov for several helpful comments that improved the manuscript.


\begin{thebibliography}{100}

\bibitem{MR1756138}
A.~Aldroubi and K.~Gr\"{o}chenig.
\newblock Beurling-{L}andau-type theorems for non-uniform sampling in shift-invariant spline spaces.
\newblock {\em J. Fourier Anal. Appl.}, 6(1):93--103, 2000.

\bibitem{MR10576141a}
A.~Beurling.
\newblock {B}alayage of {F}ourier--{S}tieltjes transforms.
\newblock In L.~Carleson, P.~Malliavin, J.~Neuberger, and J.~Wermer, editors, {\em The collected works of {A}rne {B}eurling. {V}ol. 2}, Contemporary Mathematicians, pages xx+389. Birkh\"{a}user Boston, Inc., Boston, MA, 1989.
\newblock Harmonic analysis.

\bibitem{MR2264211}
H.~G. Feichtinger and F.~Luef.
\newblock Wiener amalgam spaces for the fundamental identity of {G}abor analysis.
\newblock {\em Collect. Math.}, pages 233--253, 2006.

\bibitem{MR1843717}
K.~Gr\"ochenig.
\newblock {\em Foundations of time-frequency analysis}.
\newblock Applied and Numerical Harmonic Analysis. Birkh\"auser Boston, Inc., Boston, MA, 2001.

\bibitem{MR3336091}
K.~Gr\"{o}chenig, J.~Ortega-Cerd\`a, and J.~L. Romero.
\newblock Deformation of {G}abor systems.
\newblock {\em Adv. Math.}, 277:388--425, 2015.

\bibitem{grs}
K.~Gr\"{o}chenig, J.~L. Romero, and J.~St\"{o}ckler.
\newblock Sampling theorems for shift-invariant spaces, {G}abor frames, and totally positive functions.
\newblock {\em Invent. Math.}, 211(3):1119--1148, 2018.

\bibitem{MR4047939}
K.~Gr\"ochenig, J.~L. Romero, and J.~St\"ockler.
\newblock Sharp results on sampling with derivatives in shift-invariant spaces and multi-window {G}abor frames.
\newblock {\em Constr. Approx.}, 51(1):1--25, 2020.

\bibitem{MR3053565}
K.~Gr\"{o}chenig and J.~St\"{o}ckler.
\newblock Gabor frames and totally positive functions.
\newblock {\em Duke Math. J.}, 162(6):1003--1031, 2013.

\bibitem{MR3762092}
K.~Hamm and J.~Ledford.
\newblock On the structure and interpolation properties of quasi shift-invariant spaces.
\newblock {\em J. Funct. Anal.}, 274(7):1959--1992, 2018.

\bibitem{MR2744776}
C.~Heil.
\newblock {\em A basis theory primer}.
\newblock Applied and Numerical Harmonic Analysis. Birkh\"{a}user/Springer, New York, expanded edition, 2011.

\bibitem{Jia1991}
R.-Q. Jia and C.~A. Micchelli.
\newblock Using the refinement equations for the construction of pre-wavelets {II}: Powers of two.
\newblock In {\em Curves and Surfaces}, pages 209--246. Elsevier, 1991.

\bibitem{MR102702}
J.-P. Kahane.
\newblock Sur les fonctions moyenne-p\'eriodiques born\'ees.
\newblock {\em Ann. Inst. Fourier (Grenoble)}, 7:293--314, 1957.

\bibitem{MR4129870}
P.~Kurasov and P.~Sarnak.
\newblock Stable polynomials and crystalline measures.
\newblock {\em J. Math. Phys.}, 61(8):083501, 13, 2020.

\bibitem{MR3338010}
N.~Lev and A.~Olevskii.
\newblock Quasicrystals and {P}oisson's summation formula.
\newblock {\em Invent. Math.}, 200(2):585--606, 2015.

\bibitem{MR3667579}
N.~Lev and A.~Olevskii.
\newblock Fourier quasicrystals and discreteness of the diffraction spectrum.
\newblock {\em Adv. Math.}, 315:1--26, 2017.

\bibitem{MR2674875}
B.~Matei and Y.~Meyer.
\newblock Simple quasicrystals are sets of stable sampling.
\newblock {\em Complex Var. Elliptic Equ.}, 55(8-10):947--964, 2010.

\bibitem{MR3482845}
Y.~F. Meyer.
\newblock Measures with locally finite support and spectrum.
\newblock {\em Proc. Natl. Acad. Sci. USA}, 113(12):3152--3158, 2016.

\bibitem{Meyer_TN}
Y.~F. Meyer.
\newblock Crystalline measures and mean-periodic functions.
\newblock {\em Trans. R. Norw. Soc. Sci. Lett.}, (2):5--30, 2021.

\bibitem{ou}
A.~Olevskii and A.~Ulanovskii.
\newblock {\em Functions with Disconnected Spectrum}.
\newblock American Mathematical Society, 2016.

\bibitem{MR4206541}
A.~Olevskii and A.~Ulanovskii.
\newblock Fourier quasicrystals with unit masses.
\newblock {\em C. R. Math. Acad. Sci. Paris}, 358(11-12):1207--1211, 2020.

\bibitem{MR2439002}
A.~Olevski\u{\i} and A.~Ulanovskii.
\newblock Universal sampling and interpolation of band-limited signals.
\newblock {\em Geom. Funct. Anal.}, 18(3):1029--1052, 2008.

\bibitem{MR1802924}
H.~Reiter and J.~D. Stegeman.
\newblock {\em Classical harmonic analysis and locally compact groups}, volume~22 of {\em London Mathematical Society Monographs. New Series}.
\newblock The Clarendon Press, Oxford University Press, New York, second edition, 2000.

\bibitem{MR4782146}
J.~L. Romero, A.~Ulanovskii, and I.~Zlotnikov.
\newblock Sampling in the shift-invariant space generated by the bivariate {G}aussian function.
\newblock {\em J. Funct. Anal.}, 287(9):Paper No. 110600, 2024.

\bibitem{MR2040080}
K.~Seip.
\newblock {\em Interpolation and sampling in spaces of analytic functions}, volume~33 of {\em University Lecture Series}.
\newblock American Mathematical Society, Providence, RI, 2004.

\bibitem{uz}
A.~Ulanovskii and I.~Zlotnikov.
\newblock Sampling in quasi shift-invariant spaces and {G}abor frames generated by ratios of exponential polynomials.
\newblock {\em Math. Ann.}, 2024. https://doi.org/10.1007/s00208-024-03011-7

\bibitem{MR379387}
A.~J. Van~der Poorten and R.~Tijdeman.
\newblock On common zeros of exponential polynomials.
\newblock {\em Enseign. Math. (2)}, 21(1):57--67, 1975.

\bibitem{MR1503035}
N.~Wiener.
\newblock Tauberian theorems.
\newblock {\em Ann. of Math. (2)}, 33(1):1--100, 1932.

\bibitem{MR1836633}
R.~M. Young.
\newblock {\em An introduction to nonharmonic {F}ourier series}.
\newblock Academic Press, Inc., San Diego, CA, first edition, 2001.

\end{thebibliography}
 \end{document}